\documentclass[12pt,reqno]{article}
\usepackage{amsfonts}
\usepackage{amssymb,amscd,amsmath,amsthm,color}\usepackage{hyperref}
\newtheorem{theorem}{Theorem}
\newtheorem{lemma}[theorem]{Lemma}

\usepackage{enumerate,amssymb}
\newtheorem{ex}[theorem]{Example}

\newtheorem{conjecture}[theorem]{Conjecture}
\theoremstyle{remark}

\numberwithin{equation}{section}

\begin{document}
\baselineskip=15pt

\title{Fischer type determinantal inequalities for accretive-dissipative matrices}

\author{Minghua Lin} 

\date{}

\maketitle

\begin{abstract}
\noindent   Let $A=\begin{bmatrix}
A_{11} &A_{12}   \\A_{21}  & A_{22}
\end{bmatrix}$ be an $n\times n$ accretive-dissipative matrix, $k$ and $l$ be the orders of $A_{11}$ and $A_{22}$, respectively, and let $m=\min\{k,l\}$.  Then $$|\det A|\le a|\det A_{11}|\cdot|\det A_{22}|,$$ where $a=\left\{\begin{array}{l l}
    2^{3m/2}, &   \text{if}~~ m\le n/3;\\
    2^{n/2}, &  \text{if}~~ n/3<m\le n/2.\\
  \end{array} \right.$ This improves a result of Ikramov.
\end{abstract}

{\small\noindent

\noindent
AMS subjects classification 2010:  15A45.}

 \section{Introduction}
\noindent
     Let $\mathbb{M}_{n}(\bold C)$ be the set of $n\times n$ complex matrices. For any $A\in\mathbb{M}_n(\bold C)$, $A^*$ stands for the conjugate transpose of $A$.   $A\in\mathbb{M}_{n}(\bold C)$ is accretive-dissipative if it can be written as  \begin{eqnarray} \label{e0} A=B+iC, \end{eqnarray} where $B=\frac{A+A^*}{2}$ and $C=\frac{A-A^*}{2i}$ are both (Hermitian) positive definite. Conformally partition $A, B, C$ as  \begin{eqnarray} \label{e1} \begin{bmatrix}
A_{11} &A_{12}   \\A_{21}  & A_{22}
\end{bmatrix}= \begin{bmatrix}
B_{11} &B_{12}   \\B_{12}^*  & B_{22}
\end{bmatrix} +i \begin{bmatrix}
C_{11} &C_{12}   \\C_{12}^*  & C_{22}
\end{bmatrix}  \end{eqnarray}
such that all diagonal blocks are square. Say $k$ and $l$ ($k,l>0$ and $k+l=n$) the order of $A_{11}$ and $A_{22}$, respectively, and let $m=\min\{k,l\}$. 

If $A$ is positive definite and partitioned as in (\ref{e1}), then the famous Fischer determinantal inequality {(FDI)} \cite[p. 478]{HJ85} states that
\begin{eqnarray} \label{fischer det}  \det A\le \det A_{11}\cdot\det A_{22}.\end{eqnarray}

Determinantal inequalities for accretive-dissipative matrices were first investigated by Ikramov \cite{Ikr04}, who obtained:
\begin{theorem}\label{th1} Let $A\in \mathbb{M}_{n}(\bold C)$ be accretive-dissipative and partitioned as in (\ref{e1}). Then  \begin{eqnarray}\label{ikramov-det}|\det A|\le 3^m|\det A_{11}|\cdot|\det A_{22}|.
\end{eqnarray}\end{theorem}
A reverse direction to that of Theorem \ref{th1} has been given in \cite{Lin12b}. We call this kind of inequalities the Fischer type determinantal inequality for accretive-dissipative matrices. In this paper, we intend to give an improvement of (\ref{ikramov-det}). Our main result can be stated as \begin{theorem}\label{th2} Let $A\in \mathbb{M}_{n}(\bold C)$ be accretive-dissipative and partitioned as in (\ref{e1}). Then  \begin{eqnarray}\label{lin-det} |\det A|\le a|\det A_{11}|\cdot|\det A_{22}|,\end{eqnarray} where $a=\left\{\begin{array}{l l}
    2^{3m/2}, &   \text{if}~~ m\le n/3;\\
    2^{n/2}, &  \text{if}~~ n/3<m\le n/2.\\
  \end{array} \right.$
\end{theorem}
As $a<3^m$, it is clear that Theorem \ref{th2} improves Theorem \ref{th1}. The proof of Theorem \ref{th2} is given in Section 3.
 \section{Auxiliary results}
\noindent
In this section, we present some lemmas that are needed in the proof of our main result.

\begin{lemma}\cite[Property 6]{GI05} \label{lem1} Let $A\in \mathbb{M}_{n}(\bold C)$ be accretive-dissipative and partitioned as in (\ref{e1}). Then $A/A_{11}:=A_{22}-A_{21}A_{11}^{-1}A_{12}$, the Schur complement of $A_{11}$ in $A$, is also accretive-dissipative. \end{lemma}

\begin{lemma}\cite[Lemma 1]{Ikr04}  \label{lem2}  Let $A\in \mathbb{M}_{n}(\bold C)$ be accretive-dissipative as in (\ref{e0}). Then $A^{-1}=E-iF$ with $E=(B+CB^{-1}C)^{-1}$ and $F=(C+BC^{-1}B)^{-1}$. \end{lemma}

\begin{lemma}\cite[Lemma 3.2]{Zhan96}  \label{lem3}  Let $B, C\in \mathbb{M}_{n}(\bold C)$ be Hermitian and assume $B$ is positive definite. Then  \begin{eqnarray}\label{e2} B+CB^{-1}C\ge 2C.
\end{eqnarray} \end{lemma}

Here we adopt the convention that, for two Hermitian matrices $X, Y$ of the same size,  $X>(\ge) Y$  means $X-Y$ is positive (semi)definite. Of course, we do not distinguish $Y<(\le) X$ from $X>(\ge) Y$.

\begin{lemma}   \label{lem4}  Let $B, C\in \mathbb{M}_{n}(\bold C)$ be positive semidefinite. Then  \begin{eqnarray}\label{e3} |\det(B+iC)|\le \det(B+C)\le 2^{n/2}|\det(B+iC)|.
\end{eqnarray} \end{lemma}

\begin{proof}The first inequality follows from \cite[Theorem 2.2]{Zhan00} while the second one follows from \cite[Theorem 1.1]{BK09}. Here we provide a direct proof of (\ref{e3}) for the convenience of readers.  We may assume $B$ is positive definite, the general case is by a continuity argument. Let $\lambda_j$, $j=1,\ldots, n$, be the eigenvalues of $B^{-1/2}CB^{-1/2}$, where $B^{1/2}$ means the unique positive definite square root of $B$. Then \[|1+i\lambda_j|\le 1+\lambda_j\le \sqrt{2}|1+i\lambda_j|, ~~ j=1,\ldots, n.\] Also, we denote the identity matrix by $I$.

 Compute \begin{eqnarray*} |\det(B+iC)|&=&\det B\cdot|\det(I+iB^{-1/2}CB^{-1/2})|\\&=&\det B\cdot\prod_{j=1}^n|1+i\lambda_j|\\&\le& \det B\cdot\prod_{j=1}^n(1+\lambda_j)\\&=&\det B\cdot \det(I+B^{-1/2}CB^{-1/2})\\&=&\det(B+C).
\end{eqnarray*} This proves the first inequality. To show the other, compute \begin{eqnarray*} \det(B+C)&=&\det B\cdot \det(I+B^{-1/2}CB^{-1/2})\\&=&
\det B\cdot\prod_{j=1}^n(1+\lambda_j)\\&\le&\det B\cdot\prod_{j=1}^n\sqrt{2}|1+i\lambda_j|\\&=&2^{n/2}\det B\cdot|\det(I+iB^{-1/2}CB^{-1/2})|\\&=&2^{n/2}|\det(B+iC)|.
\end{eqnarray*}  \end{proof}

 \section{Main results}
\noindent
Theorem \ref{th2} follows from the next two theorems.

\begin{theorem}\label{th3} Let $A\in \mathbb{M}_{n}(\bold C)$ be accretive-dissipative and partitioned as in (\ref{e1}). Then  \begin{eqnarray}\label{lin-det1} |\det A|\le 2^{n/2} |\det A_{11}|\cdot|\det A_{22}|.\end{eqnarray}
\end{theorem}
\begin{proof} Compute \begin{eqnarray*}|\det A|&=&|\det(B+iC)|\\&\le& \det(B+C)~~~~~~~~~~~~\hbox{(By Lemma \ref{lem4})}\\&\le& \det(B_{11}+C_{11})\cdot\det(B_{22}+C_{22})~~~~~~~~~~~~\hbox{(By FDI)}\\&\le& 2^{k/2}|\det(B_{11}+iC_{11})|\cdot 2^{l/2}|\det(B_{22}+iC_{22})|~~~~~~~~~~~~\hbox{(By Lemma \ref{lem4}))}\\&=& 2^{n/2}|\det A_{11}|\cdot|\det A_{22}|.
\end{eqnarray*}\end{proof}

\begin{theorem}\label{th4} Let $A\in \mathbb{M}_{n}(\bold C)$ be accretive-dissipative and partitioned as in (\ref{e1}). Then  \begin{eqnarray}\label{lin-det2} |\det A|\le 2^{3m/2} |\det A_{11}|\cdot|\det A_{22}|.\end{eqnarray}
\end{theorem}

\begin{proof}  We have, by Lemma \ref{lem2}, that
\begin{eqnarray*} A/A_{11}&=&A_{22}-A_{21}A_{11}^{-1}A_{12}\\&=&B_{22}+iC_{22}-(B_{12}^*+iC_{12}^*)(B_{11}+iC_{11})^{-1}(B_{12}+iC_{12})
\\&=&B_{22}+iC_{22}-(B_{12}^*+iC_{12}^*)(E_k-iF_k)(B_{12}+iC_{12})
\end{eqnarray*} with \begin{eqnarray*}\label{ef} E_k=(B_{11}+C_{11}B_{11}^{-1}C_{11})^{-1}, ~~~ F_k=(C_{11}+B_{11}C_{11}^{-1}B_{11})^{-1}.
\end{eqnarray*}

By Lemma \ref{lem3} and the operator reverse monotonicity of the inverse, we get \begin{eqnarray}\label{p1} E_k\le \frac{1}{2}C_{11}^{-1}, ~~~ F_k\le  \frac{1}{2}B_{11}^{-1}.
\end{eqnarray}

Setting $A/A_{11}=R+iS$ with $R=R^*$ and $S=S^*$. By Lemma \ref{lem1}, we know $R$ and $S$ are positive definite. A calculation shows \begin{eqnarray*}\label{r} R&=&B_{22}-B_{12}^*E_kB_{12}+C_{12}^*E_kC_{12}-B_{12}^*F_kC_{12}-C_{12}^*F_kB_{12};\\\label{s} S&=&C_{22}+B_{12}^*F_kB_{12}-C_{12}^*F_kC_{12}-C_{12}^*E_kB_{12}-B_{12}^*E_kC_{12}.
\end{eqnarray*}

It can be verified that \begin{eqnarray*}\label{p3}\pm(B_{12}^*F_kC_{12}+C_{12}^*F_kB_{12})&\le& B_{12}^*F_kB_{12}+C_{12}^*F_kC_{12};\\ \label{p4} \pm(C_{12}^*E_kB_{12}+B_{12}^*E_kC_{12})&\le& B_{12}^*E_kB_{12}+C_{12}^*E_kC_{12}.
\end{eqnarray*}
Thus, \begin{eqnarray}\label{p5} R+S\le B_{22}+2B_{12}^*F_kB_{12}+C_{22}+2C_{12}^*E_kC_{12}.
\end{eqnarray}

As $B, C$ are positive definite, we also have
\begin{eqnarray}\label{p6} B_{22}> B_{12}^*B_{11}^{-1}B_{12}, ~~\hbox{and}~~ C_{22}> C_{12}^*C_{11}^{-1}C_{12}.
\end{eqnarray}

Without loss of generality, we assume $m=l$. Compute
  \begin{eqnarray*} |\det(A/A_{11})|&=&|\det(R+iS)|\\&\le&\det(R+S)~~~~~~~~~~~\hbox{(by Lemma \ref{lem4})} \\&\le&\det(B_{22}+2B_{12}^*F_kB_{12}+C_{22}+2C_{12}^*E_kC_{12}) ~~~~~~~~~~~\hbox{(by (\ref{p5}))}
  \\&\le& \det(B_{22}+B_{12}^*B_{11}^{-1}B_{12}+C_{22}+C_{12}^*C_{11}^{-1}C_{12})~~~~~~~~~\hbox{(by (\ref{p1}))}\\&<&\det(2(B_{22}+C_{22}))~~~~~~~~~\hbox{(by (\ref{p6}))}\\&=& 2^m\det(B_{22}+C_{22})\\&\le&2^m\cdot 2^{m/2}|\det(B_{22}+iC_{22})|~~~~~~~~~~~\hbox{(by Lemma \ref{lem4})}\\&=&2^{3m/2}|\det A_{22}|.
\end{eqnarray*}
The proof is complete by noting $\det(A/A_{11})=\frac{\det A}{\det A_{11}}$.
 \end{proof}

It is natural to ask whether $a$ in (\ref{lin-det}) can be replaced by a smaller number?  There is evidence that the following could hold:

\begin{conjecture} Let $A\in \mathbb{M}_{n}(\bold C)$ be accretive-dissipative and partitioned as in (\ref{e1}). Then  \begin{eqnarray*}\label{lin-con} |\det A|\le 2^m|\det A_{11}|\cdot|\det A_{22}|.\end{eqnarray*} \end{conjecture}

We end the paper by an example showing that if the above conjecture is true, then the factor $2^m$ is optimal.
\begin{ex} Let $A=\begin{bmatrix}
(1+\epsilon)(1+i) &i-1   \\i-1 &(1+\epsilon)(1+i)
\end{bmatrix}=\begin{bmatrix}
1+\epsilon &-1   \\-1 &1+\epsilon
\end{bmatrix}+i\begin{bmatrix}
 1+\epsilon  &1  \\1& 1+\epsilon
\end{bmatrix}$ with $\epsilon>0$. Then $A$ is accretive-dissipative. As $\epsilon\to 0^+$, we have \[\frac{|\det A|}{|\det A_{11}|\cdot|\det A_{22}|}\to 2.\] \end{ex}

\section{Acknowledgement}
The author thanks the referee for his/her careful reading of the manuscript.

\vskip 10pt

\noindent
Minghua Lin

 Department of Applied Mathematics,

University of Waterloo,

 Waterloo, ON, N2L 3G1, Canada.

mlin87@ymail.com

\end{document}